\documentclass[reqno]{amsart}
\usepackage{amsmath, amssymb, amsthm, amsfonts, amsgen, stmaryrd}
\usepackage[centertags]{amsmath}
\usepackage{indentfirst}
\usepackage{fancyhdr}
\usepackage{graphicx}
\usepackage{psfrag}
\usepackage{esint,color}
\numberwithin{equation}{section}

\newcommand{\ds}{\displaystyle}
\newcommand{\R}{{\mathbb{R}}}

\newcommand{\N}{{\mathbb{N}}}
\newcommand{\dx}{\,dx}

\newcommand{\ie}{{; \it i.e., }}

\let\e=\varepsilon
\let\d=\delta

\let\a=\alpha
\let\b=\beta
\let\G=\Gamma

\newtheorem{definition}{Definition}[section]
\newtheorem{lemma}[definition]{Lemma}
\newtheorem{theorem}[definition]{Theorem}
\newtheorem{proposition}[definition]{Proposition}

\theoremstyle{definition}
\newtheorem{remark}[definition]{Remark}

\begin{document}
\title[A second-order singular perturbation model for phase transitions]
{Asymptotic analysis of a second-order singular perturbation model for phase transitions}
\author[M. Cicalese, E.N. Spadaro  and C.I. Zeppieri]{Marco Cicalese, Emanuele Nunzio Spadaro and Caterina Ida Zeppieri}

\address[Marco Cicalese]{Dipartimento di Matematica e Applicazioni ``R. Caccioppoli'', Universit\`a di Napoli Federico II, Via Cintia, 80126 Napoli, Italy}
\email[Marco Cicalese]{cicalese@unina.it}

\address[Emanuele Nunzio Spadaro]{Institut f\"ur Mathematik, Universit\"at Z\"urich, Winterthurerstrasse 190, CH-8057 Z\"urich, Switzerland}
\email[Emanuele Nunzio Spadaro]{emanuele.spadaro@math.uzh.ch}

\address[Caterina Ida Zeppieri] {Institut f\"ur Angewandte Mathematik, Universit\"at Bonn, Endenicher Alle 60, 53115 Bonn, Germany}
\email[Caterina Ida Zeppieri]{zeppieri@uni-bonn.de}

\maketitle
\begin{abstract}
We consider the problem of the asymptotic description, as $\e$ tends to zero, of the functionals $F^k_\e$ introduced by Coleman and Mizel in the theory of nonlinear second-order materials\ie
\begin{equation*}
F^k_\e(u):=\int_I \left(\frac{W(u)}{\e}-k\,\e\,(u')^2+\e^3(u'')^2\right)\dx,\quad u\in W^{2,2}(I),
\end{equation*}
where $k>0$ and $W\colon\R\to[0,+\infty)$ is a double-well potential with two potential wells of level zero at $a,b\in\R$.
By proving a new nonlinear interpolation inequality,
we show that there exists a positive constant $k_0$ such that, for $k<k_0$ and for a class of potentials $W$,
$F^k_\e$ $\G(L^1)$-converges to
\begin{equation*}
F^k(u):={\bf m}_k \, \#(S(u)),\quad u\in BV(I;\{a,b\}),
\end{equation*}
where ${\bf m}_k$ is a constant depending on $W$ and $k$.
Moreover, in the special case of the classical potential $W(s)=\frac{(s^2-1)^2}{4}$,
we provide an upper bound on the values of $k$ such that the minimizers of $F_\e^k$ cannot develop oscillations on some fine scale,
thus improving previous estimates by Mizel, Peletier and Troy.
\end{abstract}

\begin{center}
\begin{minipage}{12cm}
\small{ \vspace{15pt} \noindent {\bf Keywords}:
Second order singular perturbation, phase transitions, nonlinear interpolation, $\Gamma$-convergence.

\vspace{6pt} \noindent {\it 2000 Mathematics Subject
Classification:} 49J45, 49M25, 74B05, 76A15.}
\end{minipage}
\end{center}


\section{Introduction}
In this note we address some features of the limiting behavior of the minimizers of a class of second-order singular perturbation energies.
The model we analyze was introduced in $1984$ by Coleman and Mizel in the context of the theory of second-order materials and was then studied in \cite{CMM} in collaboration with Marcus.

Coleman and Mizel proposed a model for nonlinear materials in which the free energy depends on both first and second order spatial
derivatives of the mass density.
In this way they expected to prove the occurrence of layered structures of the ground states (as observed in concentrated soap solutions and
metallic alloys) without appealing to non-local energies (such as, for example, the Otha-Kawasaki functional \cite{OK}).
Specifically, they introduced the free-energy functional $F^k_\e\colon L^1(I)\longrightarrow (-\infty,+\infty]$ given by 
\begin{equation}\label{intro-energy-scal}
F^k_\e(u,I)=\begin{cases} \ds \int_I \left(\frac{W(u)}{\e}-k\,\e\,(u')^2+\e^3(u'')^2\right)\dx  & \text{if}\; u\in W^{2,2}(I), \cr
+\infty & \text{if}\; u\in L^1(I)\setminus W^{2,2}(I),
\end{cases}  
\end{equation}
where $u$ (the mass density) is the order parameter of the system, $\e, k>0$ and $W\colon\R\to[0,+\infty)$ is a double-well potential with two potential
wells of level zero at $a,b\in\R$.

As $\e$ goes to zero, the functional \eqref{intro-energy-scal} accounts for the energy stored by a one-dimensional physical system occupying the bounded open interval $I$.
This model can be viewed as a scaled second-order Landau expansion of the classical Cahn-Hillard model for sharp phase transition\ie
\begin{equation*}
\min\left\{\int_IW(u)\, dx:\,u\in L^1(I),\, \fint_Iu\, dx=\lambda \,a+(1-\lambda)\,b\right\}, \quad 0<\lambda<1.
\end{equation*}
For the Cahn-Hillard model the lack of uniqueness is usually solved in the context of first-order gradient theory of phase transitions
considering the simplest diffuse phase transition model\ie the Van der Waals model.
The latter is obtained by adding a singular gradient perturbation to the previous functional.
After scaling, the new functional $\mathcal{F}_\e\colon L^1(I)\to [0,+\infty]$ reads as
\begin{equation*}
\mathcal{F}_\e(u)=\begin{cases} \ds \int_I \left(\frac{W(u)}{\e}+\e(u')^2\right)\dx & \text{if}\; u\in W^{1,2}(I), \cr
+\infty & \text{if}\; u\in L^1(I)\setminus W^{1,2}(I).
\end{cases}  
\end{equation*}
If $W$ grows at least linearly at infinity, Modica and Mortola \cite{M, MM} proved that sequences $(u_\e)$ with equi-bounded energy ({\it i.e.}~such that $\sup_\e \mathcal F_\e(u_\e)<+\infty$)
cannot oscillate as, up to subsequences, they converge in $L^1(I)$ to a function $u\in BV(I;\{a,b\})$. Moreover, the $\G(L^1)$-limit of $\mathcal F_\e$ is given by
\begin{equation}\label{per}
\mathcal{F}(u)=\begin{cases} \ds {\bf m} \#(S(u))& \text{if}\; u\in BV(I;\{a,b\}), \cr
+\infty & \text{otherwise in}\; L^1(I),
\end{cases}  
\end{equation}
for a suitable constant ${\bf m}$ depending on the double-well potential $W$.

The above phenomenon characterizes first order phase transitions of every material having positive surface energy.

On the other hand, in nature there are materials that relieve energy whenever the measure of their surface is increased.
These materials have a so-called negative surface energy.
To give a mathematical description of this kind of materials within the framework of the gradient theory of phase transitions, Coleman and Mizel introduced the energy $F_\e^k$.

The requirement for an energy of the form of $F_\e^k$ to be bounded from below forces the coefficient in front of the highest gradient squared to be nonnegative. On the other hand different phenomena can occur depending on the coefficient $k$ in front of $\e\,(u')^2$.
Specifically, for negative constants $k$, different authors showed that this model leads to the same asymptotic behavior of
the first order perturbation, avoiding oscillations and converging to a sharp interface functional.
The case $k<0$ was settled by Hilhorst, Peletier and Sch\"atzle in \cite{HPS}, where the authors proved that the functionals $F^k_\e$
$\Gamma(L^1)$-converge to a limit functional of type \eqref{per}.
The case $k=0$ was instead considered by Fonseca and Mantegazza. In \cite{FM} the authors established the same limit behavior thanks to
a compactness result for sequences with equi-bounded energy obtained exploiting some {\it a priori} bounds given by the growth assumption on the double-well potential $W$ and by a Gagliardo-Nirenberg interpolation inequality.

\medskip

In this paper we investigate the case $k>0$.

The presence of a negative contribution due to the term involving the first order derivative makes the problem quite unusual in the context of higher-order models of phase transitions.

In particular, since the three different terms in the energy are of the same order, their competition's outcomes strongly depend on the value of $k$.

Heuristically, large values of $k$ make the phases highly unstable favoring oscillations between them and correspond to negative surface tensions.

This was rigorously proved by Mizel, Peletier and Troy in \cite{MPT}. The authors considered the classical potential $W(s)=\frac{1}{4}(s^2-1)^2$
and showed that, for $k>0.9481$, $\lim_{\e\to 0}\min F^k_\e=-\infty$ and that there exists a class of minimizers of $F^k_\e$ which are non-constant periodic functions oscillating between the two potential wells. Finer properties of these minimizers have been studied also in \cite{Marcus}.

It is worth mentioning here that analogous results have been obtained for the non-local perturbations of the Van der Waals model in one-dimensional space, as the already mentioned Otha--Kawasaki model (see, for example, the forerunner study of M\"uller \cite{Mu} in the context of coherent solid-solid phase transitions).
These energies, when viewed as functionals of a suitable primitive of the order parameter of the system,
become second-order functionals with a potential constraint on the first derivative, and lead to similar results.

What was left open by the analysis carried out in \cite{MPT} is the case of ``small'', positive constants $k$.
We prove here that, under the assumptions that the potential $W(s)$ is quadratic in a neighborhood of the wells and grows at least as $s^2$
at infinity (both hypothesis being necessary as discussed in Section \ref{s:cpt}),
small values of $k$ make the phases stable and correspond to positive surface tensions\ie the asymptotic behavior of $F_\e^k$ is again described by a sharp interface limit as in \eqref{per}.

The main difficulty in the achievement of the above result lies in the proof of a compactness theorem analog to the one obtained in the case of the
Modica--Mortola functional.
Indeed, the negative term in the energy $F^k_\e$ when $k>0$ gives no \textit{a priori} bounds on minimizing sequences.
Here we solve this problem showing the existence of constants $k_0,\e_0>0$ such that a new nonlinear interpolation inequality holds
(see Lemma \ref{nonli}):
\begin{equation}\label{intro-interpolation}
k_0\int_I \e (u')^2\dx \leq \int_I\left(\frac{W(u)}{\e}+\e^3(u'')^2\right)\dx,
\end{equation}
for every $u\in W^{2,2}(I)$ and $\e\leq \e_0$.
This inequality enables us to estimate from below our functionals with $F_\e^0$ (the one corresponding to $k=0$) for which the compactness result has been proved in \cite{FM}. 
Therefore, for $k<k_0$ every sequence of functions with equi-bounded energy $F^k_\e$ converges in $L^1(I)$ (up to subsequences) to a function $u\in BV(I;\{\pm 1\})$ (see Proposition \ref{comp2}).

On account of this result, we then complete the $\G$-convergence analysis of the family of functionals $F_\e^k$ by proving in Theorem~\ref{Glimit} that, for every $k<k_0$, the functionals $F^k_\e$ $\G(L^1)$-converge to 
\begin{equation}\label{intro-Glimit}
F^k(u):=\begin{cases} \ds {\bf m}_k \, \#(S(u)) & \text{if}\; u\in BV(I;\{\pm 1\}), \cr
+\infty & \text{otherwise in}\; L^1(I),
\end{cases}
\end{equation}
where ${\bf m}_k>0$ is given by the following optimal profile problem (whose solution's existence is part of the result),
\begin{multline*}
{\bf m}_k:=\min\left\{\int_\R( W(u)-k(u')^2+(u'')^2)\dx\colon u\in W^{2,2}_{\rm loc}(\R),\right.
\\ 
\left.\lim_{x\to -\infty}u(x)=-1, \, \lim_{x\to +\infty}u(x)=1 \right\}.
\end{multline*}

In the last part of the paper we address the problem of estimating the constants $k$ for which there are no oscillations
in the asymptotic minimizers.
In order to compare our results with those in \cite{MPT}, we focus here on the explicit potential $W(s)=\frac{(s^2-1)^2}{4}$ they considered.
Since the estimate we could derive for $k_0$ are very rough, in Section \ref{sec:vs} we investigate the following different problem,
\begin{equation}\label{intro-inf}
k_1:=\inf_{L>0}\;\inf\Big\{
R_{-L}^L(u), u\in W^{2,2}(-L,L)\colon u'(\pm L)=0, u'\neq 0\Big\},
\end{equation}
where for every interval $(\alpha,\beta)$ and every $u\in W^{2,2}(\alpha,\beta)$, $R_{\alpha}^{\beta}(u)$ is the Rayleigh quotient defined as
\begin{equation}\label{intro-Ray}
R_{\alpha}^{\beta}(u):=\begin{cases}
\frac{\ds\int_{\alpha}^{\beta}(W(u)+(u'')^2)\dx}{\ds\int_{\alpha}^{\beta}(u')^2\dx} & \text{if}\; \ds\int_{\alpha}^{\beta} (u')^2\dx>0,
\cr
+\infty & \text{otherwise}.
\end{cases}
\end{equation}

Problem \eqref{intro-inf} is clearly related to the computation of the optimal constant in the nonlinear interpolation inequality
\eqref{intro-interpolation} and seems to be a challenging open problem.

Clearly, for $k\geq k_1$, the minimizers exhibit an oscillating behavior. But, more importantly, we are able to show that,
for $k<\min\{k_1,1/2\}$, there are no oscillations, because of a $L^1$ compactness result in $BV(I;\{\pm 1\})$
for sequences of functions equi-bounded in energy and having at least one zero of the first derivative (see Proposition \ref{comp-ottimale}
and notice that this condition is fulfilled by any sequence of functions which is supposed to oscillate).
This compactness result, although analogous to the previous one provided for $k<k_0$, is actually more difficult, because in this last
case the energy is not everywhere positive, but can be in principle negative in a boundary layer (see Lemma \ref{interp-bordo}).

The benefit of this more refined compactness is that we can provide an upper bound and a lower bound on $k_1$ which have the same order
of magnitude.
The lower bound we obtain for $k_1$ follows by carefully tracing the constants in the linear interpolation inequality and amounts to $1/8$;
the upper bound $k_1< 0,6846$ follows, instead, from a test with quadratic polynomials and gives an improved estimate with respect the one
given in \cite{MPT} ($k_1\leq 0,9481$).

What the present analysis does not still settle is a better understanding of the interpolation constants $k_0$ and $k_1$. We conjecture, indeed, that
for every $k<k_1$ the functionals $F_\e^k$ do not develop microstructures and $\G$-converge to a sharp interface functional of type \eqref{per} up to an additive constant depending on the presence of possible boundary layers' energies.

Similarly, we plan to address the analogous analysis in any space dimension in a future work (see Remark \ref{+D} for the proof of the
compactness in any dimension as a consequence of the one-dimensional result).

%
%
%
%

\section{Notation and preliminaries}
In this section we set our notation and we recall some preliminary results we employ in the sequel.

With $I\subset \R$ we always denote an open bounded interval and with $\e,k$ two positive constants.
Moreover, we fix a class of double-well potentials with the following properties:
$W:\R\to[0,+\infty)$ is continuous, $W^{-1}(\{0\})=\{\pm1\}$ (the location of the wells clearly can be fixed arbitrarily),
and satisfies
\begin{itemize}
\item[(w)] there exists $c>0$ such that $W(s)\geq c\,(s\mp1)^2$ for $\pm s\geq0$.
\end{itemize}
Note that in particular the standard double-well potential $W(s)=\frac{(s^2-1)^2}{4}$ belongs to this class.

We consider the functionals $F_\e^k$ defined in \eqref{intro-energy-scal} and,
whenever the domain of integration is clear from the context, we simply write $F_\e^k(u)$ in place of $F_\e^k(u,I)$.
We denote by $E_\e=F_\e^0$ the functional introduced in \cite{FM}; that is
\begin{equation*}
E_\e(u,I):=
\begin{cases} 
\ds \int_I \left(\frac{W(u)}{\e}+\e^3(u'')^2\right)\dx & \text{if}\; u\in W^{2,2}(I) ,\\
+\infty & \text{if}\; u\in L^1(I)\setminus W^{2,2}(I).
\end{cases}  
\end{equation*}

As we heavily use it in the sequel, we recall here the statement of one of the main results of \cite{FM} (see \cite[Proposition 2.7]{FM}).
\begin{proposition}\label{p:cptFM}
Let $(u_{\e})\subset W^{2,2}(I)$ satisfy $\sup_{\e}E_{\e}(u_{\e},I)<+\infty$. Then, there exist a
subsequence (not relabeled)
and a function $u\in BV(I,\{\pm 1\})$ such that $u_{\e}\to u$ in $L^{1}(I)$.
\end{proposition}

We also recall two classical interpolation inequalities
(see \cite[Theorem~1.2 and (1.22) pag.~ 10]{KZ} and \cite{Ga, N}).
\begin{proposition}\label{standard-interp}
For every $a,b\in\R$, $a<b$, and every function $u\in W^{2,2} (a,b)$, the following inequalities hold:

(i) (optimal constant) 
\begin{equation}\label{e:standard-int}
\|u'\|_{L^2(a,b)} \leq c \,\| u''\|_{L^2(a,b)}+k(c)\, \|u\|_{L^2(a,b)},
\end{equation}
for every $c>0$, with $k(c)=\frac{1}{c}+\frac{12}{(b-a)^2}$;

(ii) there exists a constant $c>0$ such that 
\begin{equation}\label{e:standar-int2}
\|u'\|_{L^{\frac{4}{3}}(a,b)} \leq c\,\Big(\|u\|^{\frac{1}{2}}_{L^1(a,b)} \| u''\|^{\frac{1}{2}}_{L^2(a,b)}+ \|u\|_{L^1(a,b)}\Big).
\end{equation}
\end{proposition}

Finally, we prove the following interpolation inequality with boundary terms.

\begin{proposition}[Interpolation with boundary terms]\label{bdry-interp}
For every $a,b\in\R$ with $a<b$, $u\in W^{2,2}(a,b)$ and $c>0$, it holds
\begin{equation}\label{e:interpol}
c\int_a^b(u')^2\leq c^3\int_{a}^b(u'')^2+
\int_a^b \frac{(u\pm1)^2}{c}+(c\,u'(b)+u(b)\pm1)^2-(c\,u'(a)+u(a)\pm1)^2.
\end{equation}
\end{proposition}
\begin{proof}
We have the following identity
\begin{multline}\label{identity}
c^2(u')^2+(c^2 u''+c u'+u\pm1)^2
\\=c^4(u'')^2+(u\pm1)^2+2c(c u'+u\pm1)(c u''+u').
\end{multline}
Then, integrating both sides of \eqref{identity} we find
\begin{multline*}
\int_a^b c^2(u')^2\dx+\int_a^b (c^2 u''+c u'+u\pm1)^2\dx\\
=\int_a^b(c^4(u'')^2+(u\pm1)^2)\dx+c\big((c u'(b)+u(b)\pm1)^2-(c u'(a)+u(a)\pm1)^2\big).
\end{multline*}
Hence, dividing by $c>0$ we get the thesis.
\end{proof}

\section{Compactness}\label{s:cpt}
In this section we prove one of the main result of this paper, namely the existence of
a constant $k_{0}>0$ such that, for every $k<k_0$, the functional $F_{\e}^k$ satisfy
the same compactness property of Proposition \ref{p:cptFM}.
As an easy consequence, we then obtain the existence of the solution to the optimal profile problem for $F_\e^k$.

\subsection{Nonlinear interpolation and compactness}
In this subsection we prove a nonlinear version of the standard $L^2$-interpolation inequality of type (i) Proposition~\ref{standard-interp}.

\begin{lemma}[Nonlinear interpolation]\label{nonli}
There exists a constant $k_0>0$ such that
\begin{equation}\label{interpolation}
k_0\int_a^b (u')^2\dx \leq \frac{1}{(b-a)^2}\int_a^b W(u)\dx+(b-a)^2\int_a^b(u'')^2\dx,
\end{equation}
for every $u\in W^{2,2}(a,b)$ and for every $a,b \in \R$ with $a< b$.
\end{lemma}

\begin{proof}
Up to translations and rescalings, it is enough to prove \eqref{interpolation} for $(a,b)=(0,1)$.
To this end, we set
$$
m:=\int_0^1 u'\dx=u(1)-u(0),
$$
and, by the symmetry of \eqref{interpolation}, up to exchanging $u$ with $-u$, we assume that $m\geq 0$.
From the fundamental theorem of calculus, it follows that,
\begin{equation}\label{l1}
|u'-m|\leq \int_0^1 |u''|\dx,
\end{equation}
and hence
\begin{equation*}
\int_0^1 (u')^2\dx \leq 2 \int_0^1 (u'')^2\dx+2\,m^2.
\end{equation*}
Therefore, to prove \eqref{interpolation} it is enough to show the existence of a constant $c>0$ such that
\begin{equation}\label{enough}
m^2\leq c \int_0^1 \Big(W(u)+(u'')^2\Big)\dx.
\end{equation}
If $m^2\leq \frac{1}{2}\int_0^1 (u'')^2\dx$, then \eqref{enough} clearly follows.
If this is not the case, applying Jensen's inequality in \eqref{l1}, we get
\begin{equation}\label{e:u' uniform}
\frac{m}{2} \leq u' \leq \frac{3}{2}m.
\end{equation}
This implies that $u$ is strictly increasing in $(0,1)$ and, therefore, $u$ does not vanish in at least one of the two intervals $(0,1/2)$ and $(1/2,1)$.
Without loss of generality, we may assume that $u>0$ in $(0,1/2)$.
Hence, by \eqref{e:standard-int} and hypothesis (w), we have
\begin{equation}\label{standard_int}
\int_0^{\frac{1}{2}} (u')^2\dx \leq c \int_0^\frac{1}{2} ((u-1)^2+(u'')^2)\dx\leq c \int_0^1 (W(u)+(u'')^2)\dx.
\end{equation}
Since \eqref{e:u' uniform} implies $m^2 \leq 8 \int_0^{\frac{1}{2}} (u')^2\dx$,
from \eqref{standard_int} we get \eqref{enough} and thus the thesis.
\end{proof}

\begin{remark}\label{interp-on-R}
Dividing $\R$ into disjoint intervals of length $1$ and applying \eqref{interpolation} we may deduce
\begin{equation}\label{e:interp-on-R}
k_0\int_\R (u')^2\dx \leq \int_\R({W(u)}+(u'')^2)\dx,
\end{equation}
for every $u\in W^{2,2}_{\rm loc}(\R)$ with $k_0>0$ as in Lemma \ref{nonli}.
\end{remark}

Now we prove that Lemma \ref{nonli} together with a simple decomposition argument yield a lower bound for $F_\e^k$ in terms of the functional $E_\e$.

\begin{proposition}\label{p:stima}
For every interval $I$ and $\d>0$, there exists $\e_0>0$ such that, for every $k>0$, $\e\leq \e_0$, 
and $u\in L^1(I)$, 
\begin{equation}\label{stima}
\left(1-\frac{k}{k_0}-\delta\right)E_\e(u,I) \leq F_\e^k(u,I).
\end{equation}
\end{proposition}

\begin{proof}
A change of variable gives
$$
F^k_\e(u,I)=\int_{I/\e} (W(u)-k(u')^2+(u'')^2)\dx,
$$
where $I/\e=\{x\in \R: \e\,x\in I\}$.
Set $n_\e:= \big[\frac{|I|}{\e}\big]$; we divide the interval $I/\e$ into $n_\e$ pairwise disjoint open intervals $I_\e^i$, $i=1,\ldots,n_\e$, of length $\frac{|I|}{\e\, n_\e}$. Then, by applying \eqref{interpolation} on each interval $I_\e^i$ we get
\begin{align*}
F^k_\e (u,I)&=\sum_{i=1}^{n_\e} \int_{I_\e^i}(W(u)-k(u')^2+(u'')^2)\dx
\\
&= \Big(1-\frac{k}{k_0} \frac{\e^2 n_\e^2}{|I|^2}\Big) \int_{I/\e} W(u)\dx +\Big(1-\frac{k}{k_0} \frac{|I|^2}{\e^2 n_\e^2}\Big) \int_{I/\e} (u'')^2\dx.
\end{align*}
Since
$$
\lim_{\e\to 0}\frac{\e n_\e}{|I|}=1,
$$
we get the thesis just by unscaling.
\end{proof}

The following compactness result is now an immediate consequence of Proposition \ref{p:stima} and Proposition \ref{p:cptFM}.

\begin{proposition}[Compactness]\label{comp2}
Let $k<k_0$ and let $(u_\e)\subset W^{2,2}(I)$ be a sequence satisfying $\sup_\e F^k_\e(u_\e)<+\infty$. Then there exist a subsequence (not relabeled) and a function $u\in BV(I;\{\pm 1\})$ such that $u_\e\to u$ in $L^1(I)$.
\end{proposition}

\begin{remark}\label{+D}
Proposition \ref{p:stima} can be easily generalized to any space dimension $n$.
Namely, an immediate consequence of it is the existence of a constant $k_n\geq k_0/n>0$ such that,
for every smooth bounded domain $\Omega\subset \R^n$, $u\in W^{2,2}(\Omega)$, $k<k_n$, and $\e$ small
\begin{equation}
k\int_{\Omega}\e\,|\nabla u|^2\dx \leq
\int_\Omega \left(\frac{W(u)}{\e}+\e^3\,|\nabla^2 u|^2\right) \dx.
\end{equation}
Indeed, by a standard covering argument it is enough to discuss the case of a rectangle
$\Omega=I_1\times\cdots \times I_n$ and then the conclusion follows by an easy application of Fubini's Theorem. 
Let $\hat I_i=I_1\times\cdots\times I_{i-1}\times I_{i+1}\times\cdots \times I_n$ and $k=\frac{k_0-\d}{n}$, we have
\begin{align*}
\int_\Omega\frac{k_0-\d}{n}\,\e\,|\nabla u|^2 &=\sum_{i=1}^n \int_{\hat I_i}\int_{I_i}\frac{k_0-\d}{n}\,\e\,|\partial_i u|^2\dx_id\hat{x}_i\\
&\stackrel{\eqref{stima}}{\leq} \sum_{i=1}^n \int_{\hat I_i}\int_{I_i}\left(\frac{W(u)}{n\,\e}+\frac{\e^3}{n}\,|\partial_{ii} u|^2 \right)\dx_id\hat{x}_i\\
&\leq \int_{\Omega}\left(\frac{W(u)}{\e}+\e^3\,|\nabla^2 u|^2\right)\dx.
\end{align*}
\end{remark}

\medskip

Here we briefly comment on the assumption (w) on the double-well potential $W$.
We show with two explicit examples that the two following conditions
\begin{itemize}
\item[(i)] $\liminf_{|s|\to+\infty}\frac{W(s)}{s^2}>0$,
\item[(ii)] $\liminf_{s\to0}\frac{W(\pm1+s)}{s^2}>0$,
\end{itemize}
(which together are equivalent to (w)) are necessary to establish \eqref{stima}.

Indeed, let $l,\alpha>0$ be two parameters to be fixed later and such that $(6\,l\,\e)^{-1}\in\N$.
Consider the two families of periodic functions, of period $(6\,l\,\e)^{-1}$, $(u_\e)$ and $(v_\e)$ defined in $(0,1)$ in the following way.
On a half period, both $u_\e$ and $v_\e$ are defined by a line of slope $\alpha/\e$, an arc of parabola, and another line
of slope $-\alpha/\e$, as in the Figure \ref{f:ex}; moreover, $u_{\e}(0)=0$ and $v_\e(0)=1$ .

For the sake of simplicity, to shorten the present computation,
assume that $W$ is monotone on the intervals $(-\infty,-1)$, $(-1,0)$, $(0,1)$, $(1,+\infty)$ (note that this hypothesis is not necessary).

\begin{figure}[h!]
\centering
\input{fig2.pstex_t}
\caption{The functions $u_\e$ and $v_\e$.}\label{f:ex}
\end{figure}

It is readily verified that:
\begin{itemize}
\item[(a)] $|u_\e''|=|v_\e''|\leq\frac{2\,\alpha}{l\,\e^2}$ always and $|u_\e'|=|v_\e'|=\frac{\alpha}{\e}$ on a set of measure $\frac{2}{3}$,

\smallskip

\item[(b)] $|u_\e|\leq 2\,l\,\alpha$, so that, for $l\,\alpha$ large enough, we have $W(u_\e)\leq W(2\,l\,\alpha)$,

\smallskip

\item[(c)] $|v_\e-1|\leq 2\,l\,\alpha$, so that $W(v_\e)\leq W(1+2\,l\,\alpha)$.
\end{itemize}
Now, if \eqref{stima} holds true, from estimates (a), (b), (c) it follows that
\begin{align*}
k_0\leq\frac{\ds\int_0^1\Big(\e^3\,(u_\e'')^2+\frac{W(u_\e)}{\e}\Big)\dx}{\ds\int_0^1 \e\,(u_\e')^2\dx}&\leq
\frac{\ds \e^3\,\left(\ds\frac{\ds 2\,\alpha}{\ds l\,\e^2}\right)^2+\frac{\ds W(2\,l\,\alpha)}{\e}}{\ds \e\,\left(\frac{\alpha}{\e}\right)^2\,\frac{2}{3}}=
\frac{6}{l^2}+\frac{3}{2}\,\frac{W(2\,l\,\alpha)}{\alpha^2},
\\
k_0\leq\frac{\ds \int_0^1\Big(\e^3\,(v_\e'')^2+\frac{W(v_\e)}{\e}\Big)\dx}{\ds \int_0^1 \e\,(v_\e')^2\dx }&\leq
\frac{\ds \e^3\,\left(\frac{2\,\alpha}{\ds l\,\e^2}\right)^2+\frac{\ds W(1+2\,l\,\alpha)}{\e}}{\ds \e\,\left(\frac{\alpha}{\e}\right)^2\,\frac{2}{3}}\\
=\frac{6}{l^2}+\frac{3}{2}\frac{W(1+2\,l\,\alpha)}{\alpha^2}.
\end{align*}
Then, if (i) does not hold true, taking the limit as $\alpha$ goes to $+\infty$ and then as $l$ goes to
$+\infty$ gives contradiction.
Similarly, if (ii) is not satisfied, taking the limit as $\alpha$ goes to $0$ and then
$l$ tends to $+\infty$ yields a contradiction as well.

\subsection{Optimal profile problem}
As a consequence of Lemma \ref{nonli}, we prove here the existence of a solution to the optimal profile problem for $F_\e^k$, with $k<k_0$.
Specifically, we consider the following set of functions
$$
\mathcal{A}:=\big\{f\in W^{2,2}_{\rm loc}(\R) \colon f(x)=1\; \text{if}\; x>T,\; f(x)=-1\; \text{if}\; x<-T,\; \text{for some}\; T>0\big\} 
$$
and we define
\begin{equation}\label{mk}
{\bf m}_k:=\inf\left\{\int_\R(W(f)-k(f')^2+(f'')^2)\dx\colon f\in\mathcal A\right\}.
\end{equation}
We have the following result.
\begin{proposition}[Existence of an optimal profile]\label{equiv-min}
Let $k_0$ be as in Lemma \ref{nonli}. For every $k< k_0$ the constant ${\bf m}_k$ is positive and
\begin{multline*}
{\bf m}_k:=\min\bigg\{\int_\R( W(f)-k(f')^2+(f'')^2)\dx\colon f\in W^{2,2}_{\rm loc}(\R),\\ 
\lim_{x\to -\infty}f(x)=-1, \, \lim_{x\to +\infty}f(x)=1 \bigg\}.
\end{multline*}
\end{proposition}

\noindent Before proving Proposition \ref{equiv-min}, we introduce the functions $G^k,\, H^k\colon \R^2\longrightarrow \R$ given by
\begin{multline*}
G^k(w,z):=\inf\bigg\{\int_0^1 \Big(W(g)-k(g')^2+(g'')^2\Big)\dx \colon g\in C^2([0,1]),\\
g(0)=w,\, g(1)=1,\, g'(0)=z,\, g'(1)=0  \bigg\}
\end{multline*}
and
\begin{multline*}
H^k(w,z):=\inf\bigg\{\int_0^1 \Big(W(h)-k(h')^2+(h'')^2\Big)\dx \colon h\in C^2([0,1]),\\
h(0)=-1,\, h(1)=w,\, h'(0)=0,\, h'(1)=z  \bigg\}.
\end{multline*}

If $G:=G^0$ and $H:=H^0$ are the corresponding functions for $k=0$ it is easy to check (see also \cite[Section 2]{FM}) that
$$
\lim_{(w,z)\to (1,0)} G(w,z)=\lim_{(w,z)\to (-1,0)} H(w,z)=0.
$$
Then, by the positivity of $k$ and by virtue of \eqref{interpolation} we have
$$
\Big(1-\frac{k}{k_0}\Big)G\leq G^k\leq G
\quad  \text{and} \quad 
\Big(1-\frac{k}{k_0}\Big)H\leq H^k\leq H,
$$
which lead immediately to
\begin{equation}\label{lim-aux}
\lim_{(w,z)\to (1,0)} G^k(w,z)=\lim_{(w,z)\to (-1,0)} H^k(w,z)=0\quad\forall\;k<k_0.
\end{equation}

\begin{proof}[Proof of Proposition \ref{equiv-min}]
By virtue of the nonlinear interpolation inequality of Lemma \ref{nonli}, the proof of this proposition is an easy modification of that of \cite[Lemma 2.5]{FM}.

The positivity of ${\bf m}_k$ follows from Remark \ref{interp-on-R} and \cite[Lemma 2.5]{FM}, since
$$
{\bf m}_k=\inf_{f\in\mathcal{A}}\int_\R( W(f)-k(f')^2+(f'')^2)\dx\geq \Big(1-\frac{k}{k_0}\Big)\inf_{f\in\mathcal{A}}\int_\R( W(f)+(f'')^2)\dx>0.
$$

Now we prove that ${\bf m}_k =\tilde{\bf m}_k$, where
\begin{multline*}
\tilde{\bf m}_k:=\inf\left\{\int_\R( W(f)-k(f')^2+(f'')^2)\dx\colon f\in W^{2,2}_{\rm loc}(\R),\right.
\\ 
\left.\lim_{x\to -\infty}f(x)=-1, \, \lim_{x\to +\infty}f(x)=1 \right\}.
\end{multline*}
Clearly, ${\bf m}_k \geq \tilde{\bf m}_k$. For the converse inequality, fix $\sigma >0$ and let $f$ be an admissible function for $\tilde{\bf m}_k$ such that
$$
\int_\R( W(f)-k(f')^2+(f'')^2)\dx\leq \tilde{\bf m}_k+\sigma.
$$
We show that it is possible to find two sequences $(x_j)$ and $(y_j)$ converging to $+\infty$ and $-\infty$ respectively, and such that
$$
|f'(x_j)|+|f'(y_j)|+|f(x_j)-1|+|f(y_j)+1| \to 0,
$$
as $j\to +\infty$. Indeed, in view of Remark \ref{interp-on-R} we have
$$
(k_0-k)\int_\R (f')^2\dx \leq \int_\R( W(f)-k(f')^2+(f'')^2)\dx\leq \tilde{\bf m}_k+\sigma.
$$
Thus, since $k<k_0$ we deduce that $f'\in L^2(\R)$ and so there exist two sequences of points $x_j \to +\infty$ and $y_j \to -\infty$ such that
$$
\lim_{j\to +\infty} f'(x_j)=\lim_{j\to -\infty} f'(y_j)=0.
$$
Let $g$ and $h$ be two admissible functions for $G^k(f(x_j),f'(x_j))$ and $H^k(f(y_i),f'(y_i))$, respectively, such that
$$
\int_0^1 (W(g)-k(g')^2+(g'')^2)\dx \leq G^k(f(x_j),f'(x_j))+\sigma,
$$

$$
\int_0^1 (W(h)-k(h')^2+(h'')^2)\dx \leq H^k(f(y_j),f'(y_j))+\sigma,
$$
 and set
$$
g_j(x):=g(x-x_j),\quad h_j(x):=h(x-y_j+1).
$$
We define
$$
f_j(x):=\begin{cases}
1 & \text{if}\quad x\geq x_j+1,\cr
g_j(x) & \text{if}\quad x_j\leq x \leq x_j+1,\cr
f(x) & \text{if} \quad y_j \leq x \leq x_j, \cr
h_j(x) &\text{if} \quad y_j-1 \leq x \leq y_j, \cr
-1 &\text{if}\quad x\leq y_j-1.  
\end{cases}
$$
Clearly, $f_j$ is a test function for ${\bf m}_k$ and for $k< k_0$ we have
\begin{eqnarray*}
\tilde{\bf m}_k +\sigma &\geq &\int_\R (W(f)-k(f')^2+(f'')^2)\dx \geq \int_{y_j}^{x_j}(W(f)-k(f')^2+(f'')^2)\dx\\
&=& \int_\R (W(f_j)-k(f_j')^2+(f_j'')^2)\dx - \int_{x_j}^{x_j+1} (W(g_j)-k(g_j')^2+(g_j'')^2)\dx\\
&\quad&-\int_{y_j-1}^{y_j} (W(h_j)-k(h_j')^2+(h_j'')^2)\dx\\
&\geq &{\bf m}_k-G^k(f(x_j),f'(x_j))-H^k(f(y_j),f'(y_j))-2\sigma.
\end{eqnarray*}
Hence we conclude letting $j\to +\infty$ and appealing to \eqref{lim-aux}.

Finally, it remains to prove that $\tilde{\bf m}_k$ admits a minimizer.
To this end, let $(f_n)\subset W^{2,2}_{\rm loc}(\R)$ be a sequence which realizes $\tilde{\bf m}_k$. Then, by Remark \ref{interp-on-R} we have
$$
\lim_{n\to +\infty}\Big(1-\frac{k}{k_0}\Big)\int_\R (W(f_n)+(f''_n)^2)\dx\leq \lim_{n\to +\infty}\int_\R (W(f_n)-k(f'_n)^2+(f''_n)^2)\dx=\tilde {\bf m}_k. 
$$ 
Hence, again by interpolation and appealing to the Sobolev embedding theorem, we deduce that (up to subsequence) the sequence of $C^1$ functions $(f_n)$ converges in $W^{1,\infty}_{\rm loc}(\R)$ to a $C^1$ function $f$ with
$$
\int_\R(W(f)+(f'')^2)\dx <+\infty.
$$
By \eqref{e:interp-on-R}, it follows that
\begin{equation}\label{integrability}
0\leq \int_\R(W(f)-k(f')^2+(f'')^2)\dx <+\infty.
\end{equation}
For every $T>0$, by the $W^{1,\infty}_{\rm loc}(\R)$-convergence of $(f_n)$, Fatou Lemma and the lower semicontinuity of the $L^2$-norm of the second derivative, we have
\begin{align}\label{half-line} 
\int_{-T}^T(W(f)-k(f')^2+(f'')^2)\dx &\leq \liminf_{n\to +\infty} \int_{-T}^T(W(f_n)-k(f_n')^2+(f_n'')^2)\dx\notag
\\
&\leq\liminf_{n\to +\infty} \int_\R(W(f_n)-k(f_n')^2+(f_n'')^2)\dx,
\end{align}
where the last inequality in \eqref{half-line} is a consequence of \eqref{e:interp-on-R} written for the two half lines $(-\infty,T)$ and $(T,+\infty)$. Then, taking into account \eqref{integrability} and passing to the sup on $T>0$ we get
\begin{equation}\label{min-seq}
\int_\R(W(f)-k(f')^2+(f'')^2)\dx \leq \liminf_{n\to +\infty} \int_\R(W(f_n)-k(f_n')^2+(f_n'')^2)\dx=\tilde {\bf m}_k.
\end{equation}
Thus, it remains only to show that the limit function $f$ is admissible. 
Since this is a direct consequence of the third step of the proof of \cite[Lemma 2.5]{FM}, we leave some minor details to the reader and conclude the proof.
\end{proof}

\section{$\Gamma$-convergence}
On account of the compactness result Proposition \ref{comp2}, in this section we compute the $\G$-limit of the functionals $F_\e^k$ when $k<k_0$.

\begin{theorem}\label{t:Glimit}
For every $k<k_0$, the functionals $F^k_\e$ $\G(L^1)$-converge to 
\begin{equation}\label{Glimit}
F^k(u):=\begin{cases} \ds {\bf m}_k \, \#(S(u)) & \text{if}\; u\in BV(I;\{\pm 1\}), \\
+\infty & \text{if}\;u\in L^1(I)\setminus BV(I),
\end{cases}
\end{equation}
where $\#(S(u))$ is the number of jumps of $u$ in $I$ and ${\bf m}_k$ is as in \eqref{mk}.
\end{theorem}



\begin{proof}
We divide the proof into two parts, proving the $\G$-liminf
and the $\G$-limsup inequality, respectively.

\medskip

\textit{Part I: $\G$-liminf.}
Let $u\in BV(I,\{\pm1\})$ and $(u_\e)\subset L^1$ such that $u_\e\to u$. We want to show that
\begin{equation}\label{e:liminf}
\liminf_{\e\to 0}F^k_\e(u_\e)\geq {\bf m}_k \#(S(u)).
\end{equation}
Clearly, it is enough to consider the case
$\lim_{\e\to 0}F^k_\e(u_\e)=\liminf_{\e\to 0}F^k_\e(u_\e)<+\infty$.
From \eqref{stima} we immediately deduce $\|{u''}_\e\|_{L^2(I)}\leq c \e^{-\frac{3}{2}}$, so that
\eqref{e:standar-int2} gives 
\begin{equation}\label{conv-ep-deriv}
\e\, u_\e' \to 0 \quad\text{in}\quad L^1(I). 
\end{equation}
Let $\#(S(u)):=N$, $S(u):=\{s_1,\ldots, s_N\}$ with $s_1<s_2<\ldots<s_N $, and set $\d_0:=\min\{s_{i+1}-s_i\colon i=1,\ldots N-1\}$. Fix $0<\d<\d_0/2$. Then (up to subsequences)
$$
u_\e \to u, \quad \e u'_\e \to 0 \quad \text{a.e. in}\quad B(s_i,\d),
$$
for every $i=1,\ldots, N$. Hence if we let $\sigma>0$, for every $i=1,\ldots, N$  we may find  two points $x_{\e,i}^+,x_{\e,i}^-\in B(s_i,\d)$ such that, for sufficiently small $\e>0$,
\begin{equation}\label{two-points}
|u_\e(x_{\e,i}^+)-1|<\sigma,\; |u_\e(x_{\e,i}^-)+1|<\sigma,\; |\e u'_\e(x_{\e,i}^+)|<\sigma,\; |\e u'_\e(x_{\e,i}^+)|<\sigma.
\end{equation} 
To fix the ideas, without loss of generality, suppose $x^-_{\e,i}<x^+_{\e,i}$ and set 
$$
\hat g_{\e,i}(x):= g_{\e,i}\Big(x-\frac{x_{\e,i}^+}{\e}\Big)\quad \text{and}\quad \hat h_{\e,i}(x):= h_{\e,i}\Big(x-\frac{x_{\e,i}^-+1}{\e}\Big),
$$
with $g_{\e,i}$ and $h_{\e,i}$ admissible for $G^k(u_\e(x_{\e,i}^+),\e u'_\e(x_{\e,i}^+))$ and $H^k(u_\e(x_{\e,i}^-),\e u'_\e(x_{\e,i}^-))$, respectively, and satisfying
$$
\int_0^1 \Big(W(g_{\e,i})-k(g_{\e,i}')^2+(g_{\e,i}'')^2\Big)\dx\leq G^k(u_\e(x_{\e,i}^+),\e u'_\e(x_{\e,i}^+))+\frac{\e}{2}
$$
and
$$
\int_0^1 \Big(W(h_{\e,i})-k(h_{\e,i}')^2+(h_{\e,i}'')^2\Big)\dx\leq H^k(u_\e(x_{\e,i}^-),\e u'_\e(x_{\e,i}^-))+\frac{\e}{2}.
$$
Now we suitably modify the sequence $(u_\e)$ ``far'' from each jump point $s_i$.
To this end, for every $i=1,\ldots, N$ we define on $\R$ the functions $v_{\e,i}$ as
$$
v_{\e,i}(x):=\begin{cases} 
1 & \text{if}\;  x\geq \frac{x_{\e,i}^+}{\e}+1\cr
\hat g_{\e,i}(x) & \text{if }\; \frac{x_{\e,i}^+}{\e} \leq x\leq \frac{x_{\e,i}^+}{\e}+1
\cr
u_{\e}(\e x) & \text{if }\; \frac{x_{\e,i}^-}{\e} \leq x\leq \frac{x_{\e,i}^+}{\e}
\cr
\hat h_{\e,i}(x) & \text{if }\; \frac{x_{\e,i}^-}{\e}-1 \leq x\leq \frac{x_{\e,i}^-}{\e}
\cr
-1 & \text{if }\; x\leq \frac{x_{\e,i}^+}{\e}-1.
\end{cases}
$$
Since each $v_{\e,i}$ is a test function for ${\bf m}_k$, we have
\begin{align*}
{\bf m}_k \leq{}& \int_\R (W(v_{\e,i})-k(v'_{\e,i})^2+(v''_{\e,i})^2)\dx
= \int_{\frac{x^-_{\e,i}}{\e}-1}^{\frac{x^+_{\e,i}}{\e}+1} \hspace{-0.1cm}
(W(v_{\e,i})-k(v'_{\e,i})^2+(v''_{\e,i})^2)\dx\\
\leq {}&\int_{x^-_{\e,i}}^{x^+_{\e,i}} 
\Big(\frac{W(u_\e)}{\e} -k\e (u'_\e)^2+\e^3(u''_\e)^2\Big)\dx+\\ 
&+ G^k(u_\e(x^+_{\e,i}),\e u'_\e(x^+_{\e,i}))+ H^k(u_\e(x^-_{\e,i}), \e u'_\e(x^-_{\e,i}))+\e.
\end{align*}
Then, as the intervals $(x^-_{\e,i},x^+_{\e,i})$ are pairwise disjoint for $i=1,\ldots, N$, in view of the non-negative character of $F^k_\e$ for $k<k_0$, we get 
\begin{align*}
\lim_{\e\to 0}F^k_\e(u_\e)&\geq \hspace{-0.05cm}\liminf_{\e\to 0}\sum_{i=1}^N \int_{x^-_{\e,i}}^{x^+_{\e,i}} 
\Big(\frac{W(u_\e)}{\e} -k\e (u'_\e)^2+\e^3(u''_\e)^2\Big)\dx
\\
&\geq \hspace{-0.05cm}N\, {\bf m}_k \hspace{-0.05cm}- \hspace{-0.05cm}\limsup_{\e\to 0} \sum_{i=1}^N\hspace{-0.05cm} \Big(G^k(u_\e(x^+_{\e,i}),\e u'_\e(x^+_{\e,i}))+ H^k(u_\e(x^-_{\e,i}), \e u'_\e(x^-_{\e,i}))\Big).
\end{align*}
Finally, letting $\sigma \to 0^+$, we conclude by \eqref{lim-aux}.

\medskip

\textit{Part II: $\G$-limsup.}
Let $u\in BV(I;\{\pm 1\})$ with $S(u)$ as in Part $I$, and set $s_0:=\a$, $s_{N+1}:=\beta$.
For $i=1,\ldots, N$ define $I_i:=[\frac{s_{i-1}+s_i}{2},\frac{s_{i}+s_{i+1}}{2}]$
and  $\d_0:=\min_i\{s_{i+1}-s_i\}$.

Fix $0<\d<\d_0$ and $f\in\mathcal A$ such that $f(x)=1$ if $x>T$, $f(x)=-1$ if $x<-T$, for some $T>0$, and 
$$
\int_{\R}(W(f)-k(f')^2+(f'')^2)\dx \leq {\bf m}_k +\frac{\d}{N}.
$$
Starting from this $f$ we construct a recovery sequence $(u_\e)$ for our $\G$-limit.

There exists $\e_0>0$ such that for every  $0<\e<\e_0$ we have $\frac{\d}{2\e}>T$.
For $\e<\e_0$, we define
$$
u_\e:=\begin{cases}
\ds f\Big(\frac{x-s_i}{\e}\Big) & \text{if}\; x\in I_i\; \text{and}\; [u](s_i)>0, \cr\cr 
\ds f\Big(-\frac{x-s_i}{\e}\Big) & \text{if}\; x\in I_i\; \text{and}\; [u](s_i)<0,\cr\cr
u(x) &\text{otherwise},
\end{cases}
$$
where $[u](s_i):=u(s_i)-u(s_{i-1})$, for $i=2,\ldots,N$.

It can be easily shown that $(u_\e)\subset W^{2,2}(I)$ and $u_\e \to u$ in $L^1(I)$. Moreover,
\begin{align*}
\lim_{\e\to 0} F^k_\e(u_\e)={}& \lim_{\e\to 0} \sum_{i=1}^N \int_{I_i}\Big(\frac{W(u_\e)}{\e}-k\e(u'_\e)^2+\e^3(u''_\e)^2\Big)\dx\\
={}& \lim_{\e\to 0} \Big\{ \sum_{i=1,\ldots,N \colon [u](s_i)>0} \int_{I_i/\e}(W(f(x))-k(f'(x))^2+(f''(x))^2)\dx\\
&+ \sum_{i=1,\ldots,N \colon [u](s_i)<0} \int_{I_i/\e}(W(f(-x))-k(f'(-x))^2+(f''(-x))^2)\dx\Big\}\\
\leq{}& {\bf m}_k N+\d = {\bf m}_k \#S(u)+\d, 
\end{align*}
hence we conclude by the arbitrariness of $\d>0$.
\end{proof}


\section{Phase transitions vs.~oscillations}\label{sec:vs}
Throughout the last two sections we fix $W(s)=\frac{(s^2-1)^2}{4}$.

In the spirit of Mizel, Peletier and Troy \cite{MPT}, in this section we provide a compactness result, alternative to that of Proposition \ref{comp2}, which asserts the existence of a range of values of $k$ such that sequences with equi-bounded energy $F^k_\e$, whose derivative vanishes at least in one point of $I$, do not develop oscillations.  
The reason for this new compactness result, as explained in the Introduction, is to give reasonable bounds on these values of $k$.

The key parameter for our analysis is the following 
\begin{equation}\label{kappa-1}
k_1:=\inf_{L>0}\;\inf\Big\{
R_{-L}^L(u),
u\in W^{2,2}(-L,L)\colon u'(\pm L)=0, u'\neq 0\Big\},
\end{equation}
where for every interval $(a,b)$ and every $u\in W^{2,2}(a,b)$, $R_a^b(u)$ is the Rayleigh quotient defined as
\begin{equation}\label{Ray}
R_a^b(u):=\begin{cases}
\frac{\ds\int_{a}^b(W(u)+(u'')^2)\dx}{\ds\int_{a}^b(u')^2\dx} & \text{if}\; \ds\int_a^b (u')^2\dx>0,
\cr
+\infty & \text{otherwise}.
\end{cases}
\end{equation}
It can be easily proved that $k_1$ can be equivalently rewritten as
$$
k_1=\inf_{L>0}\;\inf\left\{R_0^L(u), u\in W^{2,2}(0,L)\colon u\geq 0,\, u(0)=0,\, u'(L)=0\right\}.
$$
Clearly, for $k\geq k_1$, there are functions in the class defining $k_1$ for which the functionals $F_\e^k$ are non-positive. Therefore, there are minimizers of $F_\e^k$ developing an oscillating
structure, finer and finer as $\e$ approaches $0$. A thorough study of  oscillating minimizers has been carried out in \cite{MPT}.

On the other hand, when $k<k_1$ and the function $u\in W^{2,2}(I)$ is such that it is possible to divide its domain $I$ into subintervals $I_i:=(a_i,b_i)$ in which $u$ has constant sign, is strictly monotone, and $u(a_i)=u'(b_i)=0$, the definition of $k_1$ directly implies
\begin{equation}\label{interp-k1}
F^k_\e(u)\geq \Big(1-\frac{k}{k_1}\Big)E_\e(u).
\end{equation}
Hence, this case falls in the analysis performed in the previous sections and the development of oscillations for minimizers is ruled out, as implied by Theorem \ref{t:Glimit}.

This conclusion applies, for instance, when we prescribe homogeneous Neumann boundary conditions or periodic boundary conditions on $u$.

On the contrary, if we do not impose any boundary conditions on $u$, the estimate of the energy of $u$ in a
neighborhood of the extrema of the interval $I$ require a further investigation.

Such investigation is the main issue of this section. 
The main result asserts that for $k<\min\{k_1,1/2\}$, even without prescribing boundary
conditions, minimizers of $F^k_\e$ cannot develop an oscillatory structure.

\begin{proposition}\label{comp-ottimale}
Let $k<\min\{k_1,\frac{1}{2}\}$; let $(u_\e)\subset W^{2,2}(I)$ be a sequence such that $u_\e'=0$ at least in one point of $I$ and satisfying $\liminf_{\e\to 0}F^k_\e(u_\e)<+\infty$, then there exist a subsequence (not relabeled) and a function $u\in BV(I;\{\pm 1\})$ such that $u_\e\to u$ in $L^1(0,1)$.
\end{proposition}

The proof of Proposition \ref{comp-ottimale} is a  straightforward consequence of Proposition \ref{stima-ottimale} below and of
Proposition~\ref{p:cptFM}.

\begin{remark}
Unfortunately, at this stage it is still unclear if in Proposition \ref{comp-ottimale} taking the minimum between $k_1$ and $1/2$ is really necessary or it is a technical hypothesis.
\end{remark}

\begin{proposition}\label{stima-ottimale}
For every $k<\min\{k_1,\frac{1}{2}\}$ there exist two constants $C_k, C>0$ such that
\begin{equation}\label{e:stima-ottimale}
F^k_\e(u)\geq C_k FM_\e(u)-C,
\end{equation}
for every $\e>0$ and for every $u\in W^{2,2}(I)$ such that $u'$ vanishes at least in one point of
$I$.
\end{proposition}

The following lemma is the main ingredient in the proof of Proposition \ref{stima-ottimale}.

\begin{lemma}\label{interp-bordo}
Let $u\in W^{2,2}(I)$ and let $a, b\in I:=(\a,\b)$, $a<b$  be the smallest points in $I$ such that $u(a)=u'(b)=0$. Then, there exists $s>0$ such that, for every $\eta\in (0,1)$ we have
\begin{equation}\label{e:interp-bordo}
F^k_\e(u;(\a,b))\geq (1-2(1+\eta)k)FM_\e(u;(\a,b))-\frac{1}{\eta^s}.
\end{equation}
\end{lemma}

\begin{proof}
Let $\a<a<b<\b$, with $a, b$ as in the statement, thus $u$ is monotone in $(\a,b)$ and has constant sign in $(a,b)$. By the symmetry property of the problem, without loss of generality, we assume $u$ strictly increasing in $(\a,b)$ and $u>0$ in $(a,b)$. Therefore, we are in the following hypotheses:
$$
u>0\; \text{ in }\; (a,b),
\quad u'> 0\; \text{ in }\; (\a,b),  \quad u(a)=u'(b)=0.
$$
(see also Figure \ref{F:bordo}).

\begin{figure}[h!]
\centering \psfrag{u}{{$u$}} \psfrag{al}{$\a$}
\psfrag{a}{$a$} \psfrag{b}{{$b$}}
\includegraphics[scale=0.4]{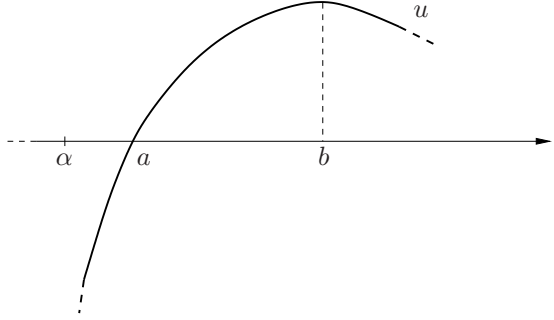}
\caption{The function $u$ in a neighborhood of $\a$.}\label{F:bordo}
\end{figure}

\noindent The interpolation inequality \eqref{e:interpol} on the interval $(\a,a)$, with $c^2=2\e^2$, gives
\begin{equation}\label{i-alfaa}
\frac{\e}{2}\int_\a^a(u')^2\dx \leq \e^3\int_\a^a (u'')^2+\int_\a^a\frac{(u+1)^2}{4\e}+\frac{\sqrt 2}{4}(\sqrt 2 \e u'(a)+1)^2,
\end{equation}
while the same computation on $(a,b)$ yields
\begin{equation}\label{i-ab}
\frac{\e}{2}\int_a^b(u')^2\dx \leq \e^3\int_a^b (u'')^2+\int_a^b\frac{(u-1)^2}{4\e}+\frac{\sqrt 2}{4}(u(b)-1)^2-\frac{\sqrt 2}{4}(\sqrt 2 \e u'(a)-1)^2.
\end{equation}
From \eqref{i-ab} we deduce
\begin{equation}\label{a-primo}
\frac{\sqrt 2}{4}(\sqrt 2 \e u'(a)-1)^2\leq \e^3\int_a^b (u'')^2-\frac{\e}{2}\int_a^b(u')^2\dx +\int_a^b\frac{(u-1)^2}{4\e}+\frac{\sqrt 2}{4}(u(b)-1)^2.
\end{equation}
Since for every $\d\in (0,1)$ we have $(A+B)^2\leq (1+\d)A^2+(1+\frac{1}{\d})B^2$, we may write
\begin{equation}\label{delta}
(\sqrt 2 \e u'(a)+1)^2 \leq (1+\d)(\sqrt 2 \e u'(a)-1)^2+4\Big(1+\frac{1}{\d}\Big).
\end{equation}
Then, gathering \eqref{i-alfaa} and \eqref{delta}, we find
\begin{multline}\label{sabato}
\frac{\e}{2}\int_\a^a(u')^2\dx \leq \e^3\int_\a^a (u'')^2+\int_\a^a\frac{(u+1)^2}{4\e}+\\
+\frac{1}{2\sqrt 2}(1+\d)(\sqrt 2 \e u'(a)-1)^2+\sqrt 2\Big(1+\frac{1}{\d}\Big).
\end{multline}
By estimating in \eqref{sabato} the quantity $(\sqrt 2 \e u'(a)-1)^2$ with \eqref{a-primo}, we get
\begin{multline*}
\frac{\e}{2}\int_\a^a(u')^2\dx\leq (1+\d)\int_\a^b \Big(\frac{(u-1)^2}{4\e}+\e^3 (u'')^2\Big)\dx+\\
-\frac{\e}{2}\int_a^b (u')^2\dx+C\,(u(b)-1)^2+\frac{C}{\d}.
\end{multline*}
Thus finally
\begin{equation}\label{en-bordo}
\frac{\e}{2}\int_\a^b(u')^2\dx\leq(1+\d)\int_\a^b \Big(\frac{(u-1)^2}{4\e}+\e^3 (u'')^2\Big)\dx+C\,u(b)^2+\frac{C}{\d}.
\end{equation}
Then, to get the thesis for $\eta=c\,\d$ for a suitable constant $c>0$, it is suffice to show that
\begin{equation}\label{basta}
u^2(b) \leq \d \int_\a^b \Big(\frac{W(u)}{\e}+\e^3 (u'')^2\Big)\dx+\frac{1}{\d} 
\end{equation} 
for every $\d\in (0,1)$. 

We prove \eqref{basta}. Consider $\nu\in(0,1)$ to be fixed later.
If $b-a\leq \nu\e$, by exploiting $u'(b)=0$ and the fundamental theorem of calculus, we have
\begin{align*}
u^2(b)&=(u(b)-u(a))^2\leq (b-a) \int_a^b (u')^2\dx \leq  {(b-a)^3}\int_\a^b (u'')^2\dx\\
& \leq {\nu^3} \e^3\int_\a^b (u'')^2\dx, 
\end{align*}
from which \eqref{basta} follows with $\d=\nu^3$.

So now suppose $b-a>\nu\e$. Again we distinguish two cases.
If
\begin{equation}\label{ip}
(u(b)-u(b-\nu\e))^2> \frac{u^2(b)}{4},
\end{equation}
then, arguing as above we find
$$
(u(b)-u(b-\nu\e))^2\leq \nu^3 \e^3\int_\a^b (u'')^2\dx,
$$
thus \eqref{ip} directly yields \eqref{basta} again with $\d=\nu^3$.

If \eqref{ip} does not hold true,
then, from $u^2(b)/2\leq (u(b)-u(b-\nu\e))^2+u^2(b-\nu\e)$,
to get \eqref{basta} it is enough to estimate $u^2(b-\nu\e)$.

By the Young Inequality $\nu^2 A^2+\frac{B^2}{\nu^2}\geq 2AB$ we have
$$
\nu^2 \int_\a^b \frac{W(u)}{\e}\dx +\frac{1}{\nu^2} \geq 2\Big(\int_\a^b \frac{W(u)}{\e}\dx\Big)^{1/2} \geq 2\Big(\int_{b-\nu\e}^b \frac{(u^2-1)^2}{4\e}\dx\Big)^{1/2}.
$$
On the other hand, using the Jensen Inequality we find
$$
\int_{b-\nu\e}^b (u^2-1)^2\dx \geq \frac{1}{\nu\e}\Big(\int_{b-\nu\e}^b (u^2-1)\dx \Big)^2.
$$
Therefore
$$
\nu^2 \int_\a^b \frac{W(u)}{\e}\dx +\frac{1}{\nu^2} \geq \frac{1}{\nu^{1/2}}\int_{b-\nu\e}^b \frac{u^2-1}{\e}\dx \geq {\nu^{1/2}}(u^2(b-\nu\e)-1),
$$
where in the last inequality we also used the fact that $b-\nu\e>a$,  and $u, u' > 0$ in $(a,b)$.
Eventually we have
$$
u^2(b-\nu\e) \leq \nu^{3/2} \int_\a^b \frac{W(u)}{\e}\dx +\frac{1}{\nu^{5/2}}+1.
$$
Taking $\d=c\,\nu^{3/2}$ for a suitable constant $c>0$, \eqref{basta} follows and thus the thesis.
\end{proof}

\begin{proof}[Proof of Proposition \ref{stima-ottimale}]
The proof is straightforward combining Lemma \ref{interp-bordo} and \eqref{interp-k1}.
Indeed, let $I=(\alpha,\beta)$: then, \eqref{e:interp-bordo} applies for two suitable neighboring intervals of $\a$ and $\b$, while \eqref{interp-k1} applies on internal intervals $(a_i,b_i)$ such that $u, u'$ have constant sign in $(a_i,b_i)$ and $u(a_i)=u'(b_i)=0$.
\end{proof}

\section{Estimates on the interpolation constant $k_1$}

\noindent In order to compare our results with those by Mizel, Peletier and Troy \cite{MPT}, in  this section we provide two estimates, one from below and one from above, on the interpolation constant $k_1$, the one from above improving their bound. 

To establish an estimate from below on $k_1$, the idea is to use the interpolation inequality (i) of Proposition \ref{standard-interp}, which gives a good bound on $k_1$ on ``large'' intervals, and to combine it with an inequality of Jensen type which is good on ``small'' intervals.  

\medskip

For every $u\in W^{2,2}(0,L)$ with $u'(L)=0$ we have
\begin{equation}\label{Jen}
\int_0^L (u')^2\dx \leq \frac{L^2}{2}\int_0^L (u'')^2\dx.
\end{equation}
Indeed, for every $x\in (0,L)$ we have
$$
|u'(x)|^2\leq \Big(\int_x^L|u''(t)|\,dt\Big)^2\leq (L-x)\int_0^L|u''(t)|^2\,dt,
$$
thus integrating on $(0,L)$ gives \eqref{Jen}. 

Then, recalling the definition of $R_0^L(u)$ \eqref{Ray}, by \eqref{Jen} we get the first bound 
\begin{equation}\label{1st-bd}
\inf_{u}R_0^L(u)\geq \frac{2}{L^2}, \quad \text{for every}\; L>0.
\end{equation}
On the other hand, Proposition \ref{standard-interp} (i) gives
\begin{equation}\label{op-int}
\int_0^L (u')^2\dx \leq 8 c^2 \int_0^L \frac{(u-1)^2}{4}\dx+2\Big(\frac{1}{c}+\frac{12}{L^2}\Big)^2\int_0^L(u'')^2\dx.  
\end{equation}
Hence gathering \eqref{Ray} and \eqref{op-int} yields the second bound
\begin{equation}\label{2nd-bd}
\inf_{u}R_0^L(u)\geq \left(\max\left\{8c^2,2\Big(\frac{1}{c}+\frac{12}{L^2}\Big)^2\right\}\right)^{-1}, \quad \text{for every}\; c,L>0.
\end{equation}
Finally, combining \eqref{1st-bd} and \eqref{2nd-bd}, recalling the definition of $k_1$ and choosing $c=1$, an explicit calculation yields
$$
k_1\geq \inf_{L>0}\max\left\{\frac{2}{L^2},\left(\max\left\{8,2\Big({1}+\frac{12}{L^2}\Big)^2\right\}\right)^{-1}\right\}=\frac{1}{8}.
$$

Concerning the estimate from above, we test the value of the Rayleigh quotient
$R_0^L$ on the quadratic polynomials $u(x)=h^2-\frac{h^2}{L^2}(x-L)^2$, with $h,L>0$.
Then, a straightforward computation gives
\begin{align*}
I_1& :=\int_{0}^LW(u(x))\dx= (1260)^{-1}\,L\,\big(128\,h^8-336\,h^4+315\big),\\
I_2& :=\int_{0}^L(u'')^2= 4\,\frac{h^4}{L^3}\quad \text{and}\quad I_3 :=\int_{0}^L(u')^2= \frac{4\,h^4}{3\,L}.
\end{align*}
Then, minimization of $\frac{I_1+I_2}{I_3}$ on $h>0$ and $L>0$ yields the bound $k_1\leq 0,6846$.

\begin{remark}
A slightly better upper bound on $k_1$ can be obtained testing $R_0^L$ on functions as in Figure \ref{f:ex}. Nevertheless, since the value we find in this way ($k_1\leq 0,6637$) is again larger than $1/2$, this does not substantially improve the statement of Proposition~\ref{stima-ottimale}, hence we omit this further computation.
\end{remark}

\medskip

\noindent{\sc Acknowledgements.}
The authors thank Andrea Braides for having drawn their attention on this problem and Sergio Conti for many interesting discussions and suggestions.

The work by M.~C. was partially supported by the European Research Council under FP7,
Advanced Grant n.~226234 ``Analytic Techniques for Geometric and Functional Inequalities".

The work by E.~N.~S. was supported by the Forschungskredit der Universit\"at Z\"urich n.~57103701.

C.~I.~Z. acknowledges the financial support of INDAM, Istituto Nazionale di Alta Matematica ``F. Severi'', through a research fellowship for Italian researchers abroad.

\end{document}